\documentclass{amsart}
\usepackage{latexsym,amssymb,amsmath,epsfig}
\usepackage{epsfig,graphicx,latexsym}
\usepackage{pstricks,pstricks-add,pst-math,pst-xkey}
\usepackage[latin1]{inputenc} 
\newtheorem{lemma}{Lemma}[section]
\newtheorem{theorem}{Theorem}[section]

\newtheorem{example}[theorem]{Example}

\begin{document}

\title[A new application of the $\otimes_h$-product to $\alpha$-labelings]{A new application of the $\otimes_h$-product to $\alpha$-labelings}

\author{S. C. L\'opez}
\address{%
Departament de Matem\`{a}tica Aplicada IV\\
Universitat Polit\`{e}cnica de Catalunya. BarcelonaTech\\
C/Esteve Terrades 5\\
08860 Castelldefels, Spain}
\email{susana@ma4.upc.edu}

\author{F. A. Muntaner-Batle}
\address{Graph Theory and Applications Research Group \\
 School of Electrical Engineering and Computer Science\\
Faculty of Engineering and Built Environment\\
The University of Newcastle\\
NSW 2308
Australia}
\email{famb1es@yahoo.es}
\date{\today}
\maketitle

\begin{abstract}
The weak tensor product was introduced by Snevily as a way to construct new graphs that admit $\alpha$-labelings from a pair of known $\alpha$-graphs. In this article, we show that this product and the application to $\alpha$-labelings can be generalized by considering as a second factor of the product, a family $\Gamma$ of bipartite $(p,q)$-graphs, $p$ and $q$ fixed. The only additional restriction that we should consider is that for every $F\in \Gamma$, there exists and $\alpha$-labeling $f_F$ with $f_F(V(F))=L\cup H$, where $L,H \subset [0,q]$ are the stable sets induced by the characteristic of $f_F$ and they do not depend on $F$. We also obtain analogous applications to near $\alpha$-labelings and bigraceful labelings.
\end{abstract}

\section{Introduction}
Let $G$ be a $(p,q)$-graph. A {\it $\beta$-labeling} of $G$ is an injective function $f:V(G)\rightarrow [0,q]$ such that the induced edge labeling $g:E(G)\rightarrow [1,q]$ defined by $g(uv)=|f(u)-f(v)|$ is also an injective function. This type of labeling, also known as {\it graceful} labeling \cite{Golomb}, was introduced by Rosa \cite{Rosa} in the context of graph decompositions. A $\beta$-labeling $f$ of $G$ is said to be an {\it $\alpha$-labeling} if there exists a constant $k$, called the {\it characteristic} of $f$, such that $\min\{f(u),f(v)\}\le k< \max\{f(u),f(v)\}$, for every edge $uv\in G$.
Clearly, a graph that admits an $\alpha$-labeling is bipartite. If $f:V(G)\rightarrow [0,q]$ is an $\alpha$-labeling of $G$, then $L=\{u: \ f(u)\le k\}$ and $H=\{u: \ f(u)>k\}$ defines a partition of $V(G)$ into two stable sets. El-Zanati et al. introduced in \cite{ElzKenEyd} the notion of near $\alpha$-labeling, which is in fact, a relaxation of an $\alpha$-labeling. A graceful labeling $f$ of $G$ is a {\it near $\alpha$-labeling} if there exists a partition $V(G)=A\cup B$ with the property that each edge of $G$ is of the form $uv$ with $u\in A$ and $v\in B$ and $f(u)<f(v)$. The introduction of this labeling is again motivated by graph decompositions.

\begin{theorem}
Let $G$ be a graph of size $n$. If $G$ admits a near $\alpha$-labeling then there exists a cyclic $G$-decomposition of $K_{2nx+1}$ for all positive integers $x$ and a cyclic $G$-decomposition of $K_{n,n}$.
\end{theorem}

A near $\alpha$-labeling can be thought as a particular case of a bigraceful labeling \cite{RinLlaSer} (see also \cite{LlaLop}), in the sense that every graph that admits a near $\alpha$-labeling also admits a bigraceful labeling. In fact, the study of cyclic decompositions of $K_{n,n}$ by a given tree of order $n$ starts at \cite{RinLlaSer}, where Ringel, Llad\'o and Serra introduced the concept of a bigraceful labeling of a graph.
Let $G$ be a bipartite graph of size $n$ and with stable sets $A$ and $B$. A pair of injective functions $f_A: A\rightarrow [0,n-1]$ and $f_B:B\rightarrow [0,n-1]$ is a {\it bigraceful labeling} \cite{LlaLop,RinLlaSer} of $G$ if the induced edge labeling on the edges $g: E(H)\rightarrow [0,n-1]$ defined by $g(uv)=f_B(v)-f_A(u)$ (with respect the ordered partition $(A,B)$) is also injective. Then, the graph $G$ is said to be a {\it bigraceful graph}.

Given two bipartite graphs $G$ and $F$ with stable sets $L_G$, $H_G$, $L_F$ and $H_F$, respectively, Snevily \cite{Snevily} defines the {\it weak tensor product} $G\bar\otimes F$ as the bipartite graph with
vertex set $(L_G\times L_F, H_G\times H_F)$ and with $(a,x)(b,y)$ being an edge if $ab\in E(G)$ and $xy\in E(F)$. Thus, it comes from the tensor product (also known as direct product) of two graphs by deleting some of its vertices and edges, according to the stable sets of the two graphs involved. However, in order to better understand the relations among the two products, that is, which vertices and edges we remove from the direct product, we orient the edges of the graphs involved and we consider the direct product of digraphs. Let $D$ and $H$ be two digraphs, the direct product of $D$ and $H$, is the digraph $D\otimes H$ (also denoted by $D\times H$) with $V(D\otimes H)=V(D)\times V(H)$ and with $((a,x),(b,y))\in E(D\otimes H)$ if and only if $(a,b)\in E(D)$ and $(x,y)\in E(H)$. Other names for the tensor product include Kronecker product, cardinal product, cross product or categorical product. A complete book on graph products is \cite{HamImrKlav11}.

\begin{lemma}\label{lemma: weak from direct product}
Let $G$ and $F$ be two bipartite graphs without isolated vertices. Let $L_G$, $H_G$, $L_F$ and $H_F$ be the stable sets of $G$ and $F$, respectively. Let $\overrightarrow{G}$ be the oriented graph obtained from $G$ by orienting every edge of $G$ from $L_G$ to $H_G$. Similarly, let $\overrightarrow{F}$ be the oriented graph obtained from $F$ by orienting every edge of $F$ from $L_F$ to $H_F$. Then, the graph $G\bar\otimes F$ is obtained from und($\overrightarrow{G}\otimes \overrightarrow{F}$) by removing all isolated vertices, where und$(D)$ denotes the underlying graph of any digraph $D$.
\end{lemma}

Snevily proves the next result.

\begin{theorem}\cite{Snevily}\label{theo: Snevily}
Let $G$ and $F$ be two bipartite graphs that have $\alpha$-labelings, with stable sets $L_G$, $H_G$, $L_F$ and $H_F$, respectively. Then, the graph $G\bar\otimes F$ also has an $\alpha$-labeling.
\end{theorem}

Using a similar proof, Theorem \ref{theo: Snevily} was extended to near $\alpha$-labelings in \cite{ElzKenEyd}.

\begin{theorem}\cite{ElzKenEyd}\label{theo: near_alpha_labelings}
If the graphs $G$ and $F$ have near $\alpha$-labelings, then so does the graph $G\bar\otimes F$ with respect to the induced vertex partitions.
\end{theorem}
Figueroa-Centeno et al. introduced in \cite{F1} a generalization of the direct product of digraphs, which has been used as a powerful tool in order to create different types of labelings for different types of graphs, see for instance \cite{ILMR,LopMunRiu1,LopMunRiu8}.
Let $D$ be a digraph and let $\Gamma $ be a family of digraphs such that $V(F)=V$, for every $F\in \Gamma$. Consider any function $h:E(D)\longrightarrow\Gamma $.
Then the product $D\otimes_{h} \Gamma$ is the digraph with vertex set being the cartesian product $V(D)\times V$ and $((a,x),(b,y))\in E(D\otimes_{h
    }\Gamma)$ \index{$\otimes_h$-product} if and only if $(a,b)\in E(D)$ and $(x,y)\in
    E(h ((a,b)))$.

In this article, we show that the weak tensor product and its application to $\alpha$-labelings and near $\alpha$-labelings can be generalized by considering as a second factor of the product, a family $\Gamma$ of bipartite $(p,q)$-graphs, with $p$ and $q$ fixed. The only additional restriction that we should consider is that for every $F\in \Gamma$, there exists and $\alpha$-labeling $f_F$ with $f_F(V(F))=L\cup H$, where $L,H\subset [0,q]$ are the stable sets defined by the characteristic of $f_F$ and they do not depend on $F$ (Theorem \ref{theo: Snevily_generalize}). We also obtain analogous applications to near $\alpha$-labelings (Theorem \ref{theo_near_alpha labeling_generalize}) and bigraceful labelings  (Theorem \ref{theo_bigraceful_generalize}).

\section{The weak $\otimes_h$-product of (di)graphs}

Inspirated by Lemma \ref{lemma: weak from direct product}, and by the definition of the $\otimes_h$-product, we introduce the weak $\otimes_h$-product of graphs.

Let $G$ be a bipartite graph with stable sets $L_G$ and $H_G$ and let $\Gamma $ be a family of bipartite graphs such that $V(F)=L\cup H$, for every $F\in \Gamma$. Consider any function $h: E(G)\rightarrow \Gamma$. Then, the product
$G\bar\otimes_h\Gamma$ is the graph with vertex set $(L_G\times L, H_G\times H)$ and $(a,x),(b,y)\in E(G\bar\otimes_h\Gamma)$ if and only if $ab\in E(G)$ and $xy\in E(h(ab))$.

\begin{example}\label{example_1}
Figure \ref{Fig_1} shows all elements $F$ in the family $\Gamma$, with $V(F)=L\cup H$ and $|E(F)|=6$, when $L=\{0,1,2\}$ and $H=\{3,4,6\}$.
\begin{figure}[h]
\begin{center}
  \includegraphics[width=46pt]{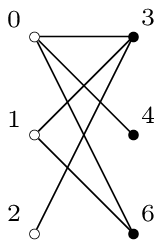}\hspace{1cm}\includegraphics[width=46pt]{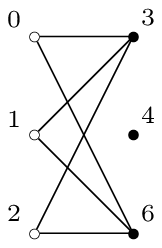}\hspace{1cm}\includegraphics[width=46pt]{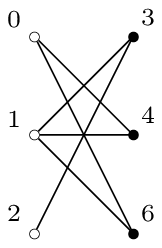}\hspace{1cm}\includegraphics[width=46pt]{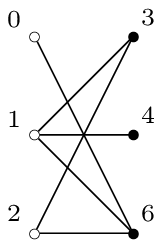}\\
  \includegraphics[width=46pt]{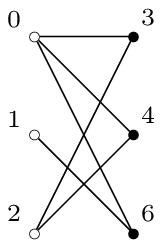}\hspace{1cm}\includegraphics[width=46pt]{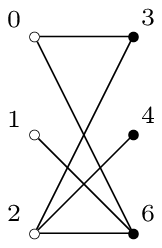}\hspace{1cm}\includegraphics[width=46pt]{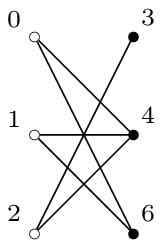}\hspace{1cm}\includegraphics[width=46pt]{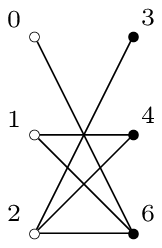}\\
\caption{All elements $F$ in $\Gamma$ with $V(F)=L\cup H$ and $|E(F)|=6$, when $L=\{0,1,2\}$ and $H=\{3,4,6\}$.}
  \label{Fig_1}
  \end{center}
\end{figure}

Let $P$ be the path defined by $V(P)=\{0,1,2\}$ and $E(P)=\{01,02\}$. Consider the function $h: E(P)\rightarrow \Gamma$ that assigns the first graph and the last graph in Fig. \ref{Fig_1} to the edges $02$ and $01$, respectively. Then, the product
$P\bar\otimes_h\Gamma$ is the graph that appears in Fig. \ref{Fig_2}, when $L_P=\{0\}$ and $H_p=\{1,2\}$.

\begin{figure}[h]
\begin{center}
  \includegraphics[width=211pt]{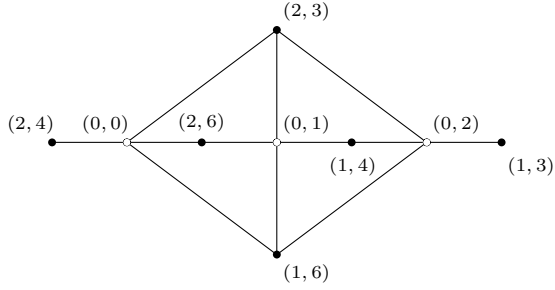}\label{Fig_2}
  \end{center}
  \caption{A graph obtained using the $\protect\bar\otimes_h$-product}
\end{figure}
\end{example}
The relation between the $\otimes_h$-product and the $\bar\otimes_h$-product is stated in the next result.
\begin{lemma}
Let $G$ be a bipartite graph with estable sets $L_G$ and $H_G$ and let $\Gamma $ be a family of bipartite graphs such that $V(F)=L\cup H$, for every $F\in \Gamma$. Assume that $G$ and every $F\in\Gamma$ do not contain isolated vertices and
let $\overrightarrow{G}$ be the oriented graph obtained from $G$ by orienting every edge of $G$ from $L_G$ to $H_G$. Similarly, for every $F\in \Gamma$, let $\overrightarrow{F}$ be the oriented graph obtained from $F$ by orienting every edge of $F$ from $L_F$ to $H_F$ and $ \overrightarrow{\Gamma}=\{\overrightarrow{F}\}_{F\in \Gamma}$. Consider any function $h:E(G)\longrightarrow\Gamma $.
Then the product
$G\bar\otimes_h\Gamma$ is the graph obtained from und($\overrightarrow{G}\otimes_{h^*}\overrightarrow{\Gamma}$), where $h^*$ is defined by $h^*((a,b))=\overrightarrow{F}$ whenever $h(ab)=F$, by removing all isolated vertices.
\end{lemma}
\subsection{Application to (near) $\alpha$-labelings}
As happens with the $\otimes_h$-product, the product $\bar\otimes_h$ that we have introduced can be used to obtain new families of labeled graphs from existing ones.

The next result generalizes Theorem \ref{theo: Snevily}.

\begin{theorem}\label{theo: Snevily_generalize}
Let $G$ be a bipartite graph that has an $\alpha$-labeling. Let $\Gamma$ be a family of bipartite graphs such that for every $F\in \Gamma$, $|E(F)|=n$ and there exists and $\alpha$-labeling $f_F$ with $f_F(V(F))=L\cup H$, where $L, H\subset [0,n]$ are the stable sets defined by the characteristic of $f_F$ and they do not depend on $F$. Consider any function $h: E(G)\rightarrow \Gamma$. Then, the graph $G\bar\otimes_h \Gamma$ also has an $\alpha$-labeling.
\end{theorem}

\begin{proof}
Let $q=|E(G)|$. Assume that every vertex of $G$ and every vertex of $F\in \Gamma$ takes the name of its label by an $\alpha$-labeling of $G$ and $F$, respectively. Consider $f:(L_G\times L, H_G\times H)\rightarrow [0, qn]$, defined by:
\begin{eqnarray*}
  f(a,x) &=& na+x, \ (a,x)\in L_G\times L,\\
  f(b,y) &=& n(b-1)+y,\ (b,y)\in H_G\times H.
\end{eqnarray*}
In what follows, we will prove that $f$ is an $\alpha$-labeling of $G\bar\otimes_h F$. Clearly, $n(b-1)+n\le n(q-1)+n\le nq$, for every $(b,y)\in H_G\times H$. Moreover, $na+x\le nk_G+k$ and $ n(b-1)+y>nk_G+k$, for every $(b,y)\in H_G\times H$ and $(a,x)\in L_G\times L$. Let us see now that $f$ is an injective function. If there exist two pairs $(b,y), (b',y')\in H_G\times H$ such that $f(b,y)=f(b',y')$, then $n(b-1)+y=n(b'-1)+y'$. That is, $y-y'\equiv 0$ (mod $n$), which implies, since $y,y'\in [1,n]$ (by definition, $y,y'\ne 0$) that $y=y'$ and thus, $b=b'$. Suppose now that there exist $(a,x), (a',x')\in L_G\times L$ such that  $f(a,x)= f(a',x')$. A similar reasoning implies that $x=x'$ and $a=a'$. Finally, we will see that the induced edge labels are distinct. Let $(a,x)(b,y), (a',x')(b',y')$ be two edges of $G\bar\otimes_h \Gamma$ such that $f(b,y)-f(a,x)=f(b',y')-f(a',x')$. That is, $n(b-1)+y-na-x=n(b'-1)+y'-na'-x'$, which is equivalent to
\begin{eqnarray}\label{eq: equality}
 n(b-a-b'+a')&=&y'-x'-y+x.
\end{eqnarray}
Since by hypothesis, $y'-x'$ and $y-x$ are elements in $[1,n]$, we necessarily obtain that $y'-x'=y-x$. Hence, using (\ref{eq: equality}), we conclude that $b-a-b'+a'=0$, that is, $b-a=b'-a'$. Therefore, since the vertices of $G$ correspond to the labels of an $\alpha$-labeling, we obtain that $ab=a'b'$, and in particular, that $xy,x'y'\in E(h(ab))$, which implies that $xy=x'y'$.
\end{proof}

Notice that Theorem \ref{theo: Snevily_generalize} significantly enlarges the class of bipartite graphs known to admit $\alpha$-labelings. This is done using an slightly modification of the product $\otimes_h$ defined by Figueroa-Centeno et al. in \cite{F1} and that proves once again to be of great help in the world of graph labelings.

Figure \ref{Fig_3} shows an $\alpha$-labeling of the graph obtained in Example \ref{example_1}.
\begin{figure}[h]
\begin{center}
  \includegraphics[width=205pt]{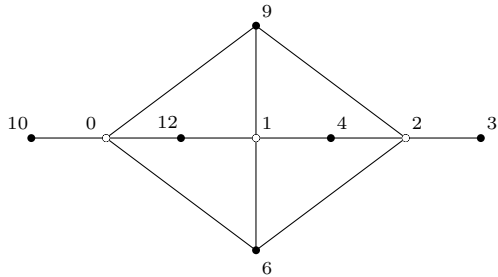}\label{Fig_3}
  \caption{An $\alpha$-labeling of the graph obtained using the $\protect\bar\otimes_h$-product}
  \end{center}
\end{figure}

With a slight modification of the previous proof, we also get the next theorem.

\begin{theorem}\label{theo_near_alpha labeling_generalize}
Let $G$ be a bipartite graph that admits a near $\alpha$-labeling. Let $\Gamma$ be a family of bipartite graphs such that for every $F\in \Gamma$, $|E(F)|=n$ and there exists and near $\alpha$-labeling $f_F$ with $f_F(V(F))=A\cup B$, where $A,B\subset [0,n]$ are the stable sets defined by $f_F$ and they do not depend on $F$. Consider any function $h: E(G)\rightarrow \Gamma$. Then, the graph $G\bar\otimes_h \Gamma$ also has a  near $\alpha$-labeling.
\end{theorem}

\begin{proof}
Let $q=|E(G)|$. Assume that every vertex of $G$ and every vertex of $F\in \Gamma$ takes the name of its label assigned by a near $\alpha$-labeling of $G$ and $F$, respectively. Suppose that every edge $xy\in E(F)$, is of the form $x\in A$, $y\in B$ and $x<y$. Consider $f:(A_G\times A, B_G\times B)\rightarrow [0, qn]$, defined by:
\begin{eqnarray*}
  f(a,x) &=& na+x, \ (a,x)\in A_G\times A,\\
  f(b,y) &=& n(b-1)+y,\ (b,y)\in B_G\times B.
\end{eqnarray*}
In what follows, we will prove that $f$ is a near $\alpha$-labeling of $G\bar\otimes_h F$. Clearly, the range of $f$ is $[0,qn]$. To see that $f$ is an injective function we have to check three possibilities. Suppose first that there exist two pairs $(a,x)\in A_G\times A$ and $(b,y)\in B_G\times B$ such that $f(a,x)=f(b,y)$. Then, $n(b-1-a)=x-y$ and thus, $n$ divides $x-y$. Since $x\in A$ and $y\in B$ this implies that $x=0$ and $y=n$. Hence, we obtain that $b-1-a=-1$, that is, $a=b$, a contradiction. The other two possibilities and the rest of the proof follow as in the proof of Theorem \ref{theo: Snevily_generalize}.
\end{proof}
\subsection{Application to bigraceful labeling}
In turns out that Theorem \ref{theo_near_alpha labeling_generalize} also holds when instead of considering graphs with near $\alpha$-labelings, we consider graphs that admit a bigraceful labeling. The fact that we are now considering a labeling with a pair of injective functions (one for each stable set) allows to simplify the labeling of the vertices of the product.

\begin{theorem}\label{theo_bigraceful_generalize}
Let $G$ be a bipartite graph that admits a bigraceful labeling. Let $\Gamma$ be a family of bipartite graphs such that for every $F\in \Gamma$, $|E(F)|=n$ and there exists a bigraceful labeling $f_F$ with $f_F(V(F))=A\cup B$, where $A,B\subset [0,n_1]$ are the stable sets defined by $f_F$ and they do not depend on $F$. Consider any function $h: E(G)\rightarrow \Gamma$. Then, the graph $G\bar\otimes_h \Gamma$ also has a bigraceful labeling.
\end{theorem}

\begin{proof}
Let $q=|E(G)|$. Assume that every vertex of $G$ and every vertex of $F\in \Gamma$ takes the name of its label assigned by a bigraceful labeling of $G$ and $F$, respectively. Suppose that every edge $xy\in E(F)$, is of the form $x\in A$, $y\in B$ and $x<y$. Consider $f:(A_G\times A, B_G\times B)\rightarrow [0, qn-1]$, defined by:
\begin{eqnarray*}
  f(a,x) &=& na+x, \ (a,x)\in (A_G\times A,B_G\times B)
\end{eqnarray*}
An easy check shows that $f$ is a bigraceful labeling of $G\bar\otimes_h \Gamma$.
\end{proof}


\begin{thebibliography}{Weak}
\bibitem{ElzKenEyd} S. I. El-Zanati, M. J. Kenig and C. Vanden Eynden, Near $\alpha$-labelings of bipartite graphs, Austral. J. Combin. {\bf 21} (2000), 275--285.
\bibitem{F1} R.M. Figueroa-Centeno, R. Ichishima, F.A. Muntaner-Batle
and M. Rius-Font, Labeling generating matrices, \emph{J. Comb. Math. and Comb. Comput.} {\bf 67} (2008), 189--216.
\bibitem{Golomb} S. W. Golomb, How to number a graph, in Graph Theory and Computing, R. C. Read, ed.,
Academic Press, New York (1972) 23--37.
\bibitem{HamImrKlav11} R. Hammarck, W. Imrich and S. Klav$\check{z}$ar, Handbook of Product Graphs, Second
Edition, CRC Press, Boca Raton, FL, 2011.
\bibitem{ILMR} R. Ichishima, S.C. L\'opez, F. A. Muntaner-Batle, M. Rius-Font, The power of digraph products applied to labelings,
{\it Discrete Math.} {\bf 312} (2012), 221--228.
\bibitem{LlaLop}  A. Llad\'o and S.C. L\'opez, Edge-decompositions of Kn,n into isomorphic copies of  a given tree, J. Graph Theory {\bf 48} (2005), 1--18.
\bibitem{LopMunRiu1}  S.C. L\'opez, F. A. Muntaner-Batle, M. Rius-Font, Bi-magic and other generalizations of super edge-magic labelings, \emph{B. Aust. Math. Soc}. {\bf 84} (2011), 137--152.
\bibitem{LopMunRiu8} S. C. L\'opez, F. A. Muntaner-Batle, M. Rius-Font, Labeling constructions using digraphs products, {\it Discrete Appl. Math.} {\bf 161} (2013), 3005--3016.
\bibitem{RinLlaSer} G. Ringel, A. Llad\'o and O. Serra, Decomposition of complete bipartite graphs into trees, DMAT Research Report 11/96, Univ. Politecnica de Catalunya.
\bibitem{Rosa} A. Rosa, On certain valuations of the vertices of a graph, {\it Theory of Graphs (Internat. Symposium, Rome, July 1966)}, Gordon and Breach, N.Y. and Dunod Paris (1967) 349--355.
\bibitem{Snevily} H. S. Snevily, New families of graphs that have $\alpha$-labelings, Discrete Math. 170 (1997), 185--194.
\end{thebibliography}
\end{document}